\pdfoutput=1

\documentclass[10pt]{amsart}                   
 \usepackage{graphicx}
\usepackage[centertags]{amsmath}
\usepackage{amsfonts}
\usepackage{amssymb}

\newtheorem{theorem}{Theorem}[section]

\newtheorem{corollary}[theorem]{Corollary}

\theoremstyle{definition}
\newtheorem{definition}[theorem]{Definition}

\newtheorem{remark}[theorem]{Remark}

\begin{document}

\newcommand{\V}{\boldsymbol{V}} 
\newcommand{\s}{\mathcal{S}}
\newcommand{\e}{\mathrm{e}^}
\newcommand{\IA}{\mathcal{I_A}}
\newcommand{\II}{\mathcal{I_I}}
\newcommand{\EI}{\mathcal{E_I}}
\newcommand{\EA}{\mathcal{E_A}}
\newcommand{\T}{\mathcal{T}}
\newcommand{\p}{\mathcal{P}}
\newcommand{\F}{\mathcal{F}}
\newcommand{\C}{\mathcal{C}}

\title{On the Potts model partition function in an external field
}

\author[L.~McDonald]{ Leslie M. McDonald}
\address{Department of Mathematics and Statistics,  University of South Alabama, Mobile, AL 36688, USA}
\curraddr{Department of Mathematics,
Oregon State University
Corvallis, OR
97331, USA}
 \email{mcdonles@math.oregonstate.edu}   

\author[I.~Moffatt]{ Iain Moffatt}
\address{Department of Mathematics and Statistics,  University of South Alabama, Mobile, AL 36688, USA}
\email{imoffatt@jaguar1.usouthal.edu}

\begin{abstract}
We study the partition function of Potts model in an external (magnetic) field, and its connections with the zero-field Potts model partition function. Using a deletion-contraction formulation for the partition function $Z$ for this model, we  show that   it can be expanded  in terms of the zero-field partition function.  We also show that $Z$ can be written as a sum over the spanning trees, and the spanning forests, of a graph $G$. Our results  extend to $Z$ the well-known spanning tree expansion for the zero-field partition function that arises though its connections with the Tutte polynomial. 
\end{abstract}

\keywords{Tutte polynomial\and Potts model \and spanning trees \and $\V$-polynomial \and external field \and Hamiltonian\and edge activities\and statistical mechanics }
\subjclass{05C31,   05C22,  82B20}
\date{\today}

\maketitle

\section{Introduction}
Let $G$ be a lattice or, more generally, a graph. The $q$-state Potts model partition function at temperature $T$ in an external field $H$  is  given by 
\[  Z(G) = \sum_{\sigma \in \s (G)}   \e{-\beta h(\sigma)},  \]
with  Hamiltonian  
\begin{equation}\label{intro1}
h(\sigma) = - J\sum_{ \{ v_i,v_j \} \in E(G) }    \delta(\sigma_i, \sigma_j)  -H \sum_{v_i\in V(G)}   \delta(1,  \sigma_i).\end{equation}
Here the $v_i$ are the vertices of $G$ and $\{ v_i,v_j \}$ its edges;  the $\sigma_i \in \{1, \ldots , q\}$ are the classical spin variables at $v_i$; $J$ is the spin-spin coupling; $\beta=1/(\kappa T)$; and the sum is over all states $\sigma$ of $G$. (A fuller description of notation can be found in  Section~\ref{s.term}.) 
The Potts model has not only proved to be invaluable to the study of phase transitions and critical phenomena in physics (see, for example, \cite{Bax82,Potts,Wu82}), but has also found applications as widely varied as tumour migration, foam behaviours, gas absorption, and social demographics (see, for example, \cite{Mey03,O+03,SG06,Sch71,Sch05,TS02}). In addition, it admits significant connections with the mathematical areas of graph theory (see, for example, \cite{BE-MPS10,Roy09,Sok05,WM00}) and quantum topology (see, for example, \cite{Jo89,Wu92}).  
Here we are interested in connections between the Potts model and graph theory, particularly with graph polynomials. It is well-known that these connections have been very fruitful, with the extra generalization provided by them leading to significant advances in both combinatorics and physics (for example, \cite{And81,Bax82,CS10,Lieb67,Sok04}).

 Connections between the Potts model and graph polynomials have traditionally focussed on the zero-field Potts model ({\em i.e.}, when $H=0$ in Equation~\eqref{intro1}), whose partition function we denote  by $Z_{zero}$.
For example, 
the classical, and well-known, relationship between the Potts model the Tutte polynomial, $T(G;x,y)$, given by
\begin{equation}\label{e.potts}Z_{zero}(G;q,v) =
q^{k(G)} v^{|V(G)|-k(G)} T(G;(q + v)/v,v + 1),\end{equation}
 where $v = e^{\beta J}-1$,
 requires the absence of an external field.  (See \cite{BE-MPS10,Bol98,Sok05,Wel93} for expositions of this connection.)  Although this relationship  has resulted in valuable interactions between combinatorics and statistical mechanics, (particularly for problems involving computational complexity and the study of the zeros of these polynomials \cite{BE-MPS10,GJ07,Roy09,Sok05,WM00}), there is a need to  extend the theory, and in particular the  deletion-contraction formulation of $Z_{zero}$ provided by its connection with the Tutte polynomial, so as to incorporate the external fields ({\em i.e.} when $H\neq 0$) required by applications of the Potts model to physics and to other areas that require models with external influences.
 
   Only very recently  have external fields been investigated in the context of graph polynomials.  In the special case of the Ising model (which is essentially the Potts model with $q=2$) an interesting new polynomial has been found in \cite{WF}.  This polynomial has a deletion-contraction relation and captures approximating functions, but does not specialize to the Tutte polynomial, nor does the deletion-contraction extend to non-constant magnetic fields. Some weighting strategies have also been developed to study non-zero magnetic fields in a graph colouring framework (see \cite{CS09,CS10,SX10}), with an accompanying graph polynomial.  These papers show that the Potts model partition function with an external field term {\em does not} have a traditional deletion-contraction reduction. It follows that $Z(G)$ can not be written as an evaluation of the Tutte polynomial $T(G)$.  
To overcome this difficulty, in \cite{EMM11}, an extension of the concept of the contraction of an edge of a vertex weighted graph was introduced. (In terms of the Potts model, the vertex weights are used to record contributions of the external field $H$, as described below.) Using this definition of contraction, 
 an extension of the Tutte polynomial, called the $\V$-polynomial was introduced, and it was shown that the Potts model partition function $Z(G)$ with a Hamiltonian given by \eqref{intro1} is an evaluation of the $\V$-polynomial. A number of consequences of this result were described in \cite{EMM11}, including a Fortuin-Kasteleyn type representation for $Z(G)$ (see also \cite{CS09,Sok99,Wu78,Wu82}), a deletion-contraction relation, and computational complexity results. This work unified an important segment of Potts model theory  bringing previously successful combinatorial machinery, including
complexity results, to bear on a wider range of statistical mechanics models.

Here we exploit the combinatorial machinery provided by this new formulation of $Z(G)$ as a graph polynomial $\V$ with a deletion-contraction-type relation to obtain a number of new expansions for $Z(G)$. 
Below, we will in fact use a slightly more general (but still standard) Hamiltonian than that exhibited in Equation~\eqref{intro1}. (We favour the more familiar \eqref{intro1} in the introduction for expositional purposes.) For our results, we  allow variable  spin-spin coupling and site-dependent external fields. This extra generality is to allow for further applications of our results, such as to the random field Ising model.  Below we  use the Hamiltonian from Equation~\eqref{e.ham2}, and we will denote the corresponding partition function by $Z_{ext}(G)$, noting that it includes $Z(G)$ as a specialization.   

We expand  the Potts model partition function in an external field, $Z_{ext}(G)$, in terms of the zero-field Potts partition function, $Z_{zero}(G)$ (see Theorem~\ref{t.exp2}). 
It is well-known that the Tutte polynomial, $T(G;x,y)$,  can  be written as a sum over spanning trees. This spanning tree expansion of the Tutte polynomial gives rise to a spanning tree expansion of $Z_{zero}(G)$. We extend this result, in Theorem~\ref{t.exp3}, to a spanning forest expansion and a spanning tree expansion for the Potts model partition function $Z(G)$ external field.
Furthermore, we show, in Theorem~\ref{t.exp1}, that  $Z_{ext}(G)$ admits the spanning tree expansion.  
This spanning tree expansion for $Z_{ext}(G)$ extends the spanning tree expansion for the zero-field Potts model partition function $Z_{zero}(G)$ that arises through its connection with the Tutte polynomial shown in Equation~\eqref{e.potts}. We conclude by specializing our results to several other common types of Hamiltonian from the physics literature. In addition to these results, a secondary purpose of this paper is to introduce the new tool of the $\V$-polynomial to the physics literature.

This paper is structured as follows: our notation is described in Section~\ref{s.term}. Full statements of our main results on $Z_{ext}(G)$ are given in Section~\ref{s.exp}, with their proofs deferred until Sections~\ref{s.trees} and \ref{s.for}. The main tool used here is the realization of $Z_{ext}(G)$ as an evaluation of the  $\V$-polynomial, which is described in Section~\ref{s.v}.

This work is based on L.M.'s Masters Thesis  which was supervised by I.M.

\section{Notation and terminology}\label{s.term}

\subsection{The $q$-state Potts model}

We begin by recalling some basic information about the Potts model.  
Let $G$ be a graph and consider a set $\{1,2,\ldots, q\}$ of $q$ elements, called \emph{spins}.  
A \emph{state} of a graph $G$ is an assignment of a single spin to each vertex
of the graph. 
Thus, if the vertex set of $G$ is $V(G)=\{v_1, \ldots, v_n\}$, then a {\em state} of $G$ is a function  $\sigma: V(G)\rightarrow \{1,\ldots , q\} $. We let $\s(G)$ denote the set of states of $G$.

The interaction energy may be thought of simply as  weights on the edges of the graph. In physics applications, the interaction energies are typically real numbers.
However, in our context of graph polynomials, they may simply be taken to be independent commuting variables. We will denote the interaction energy on an edge $e=\{v_i, v_j\}$ by $J_e =J_{i,j} = J_{v_i,v_j}$.
 
The \emph{zero field Hamiltonian} is

\begin{equation} \label{z.ham}
h(\sigma ) =  - \sum\limits_{\{i,j\}  \in E(G)} J_{i,j} {\delta(\sigma _i ,\sigma _j)},
\end{equation}
where $\sigma$ is a state of a graph $G$, where $\sigma_i $ is the spin at vertex
$v_i$,  where $\delta $ is the Kronecker delta function, and each edge $\{i,j\}$ is assigned an interaction energy $J_{i,j} $.

Here we will use a common generalization (see, for example, \cite{ACCN88,BBCK00,Sok99,SX10,Wu82}), of the Hamiltonian shown in Equation~\eqref{intro1}. 
\begin{definition} \label{general Hamiltonian}
Let $G$ be a graph.  Assign to each edge $e=\{i,j\}$ an interaction energy  $J_e=J_{i,j} $, and assign to each vertex $v_i$ a  {\em field vector} $ \boldsymbol{M}_i :=(  M_{i,1}, M_{i,2}, \ldots , M_{i,q}  )$. 
Then the  \emph{Hamiltonian of  the Potts model with variable edge interaction energy and variable  field} is 
\begin{equation}\label{e.ham2}  
h(\sigma) =  -\sum_{ \{ i,j \} \in E }  J_{i,j} \delta(\sigma_i, \sigma_j)  - \sum_{v_i\in V(G)}    \sum_{\alpha =1}^q M_{i,\alpha}  \delta(\alpha,  \sigma_i).
\end{equation}
\end{definition}
Note that in this model  the external field contributes   $M_{i,\sigma_i}\in \mathbb{C}$ to the Hamiltonian  for each vertex $v_i$ with spin $\sigma_i$. The field vectors just record the possible contributions $M_{i,j}$ at each vertex.  (The formal tool of field vectors is used here in order to naturally describe the deletion-contraction relation for the Potts model partition function in an external field.)

Note that the Hamiltonian of Equation \eqref{z.ham} may be recovered from the general Hamiltonian of Equation \eqref{e.ham2} by  taking all the  field vectors to be zero.  The Hamiltonian also subsumes other frequently studied forms of the Hamiltonian from the literature, including \eqref{intro1}. We refer the reader to \cite{EMM11} for details.

Regardless of the choice of Hamiltonian, the Potts model partition function is the sum over all possible states of an exponential function of the Hamiltonian:
\begin{definition}  Given a set of
$q$ spins and a Hamiltonian $h$,  the \emph{$q$-state Potts
model partition function} of a graph $G$ is
\begin{equation*}
Z(G) = \sum_{\sigma \in \s (G)}   \e{-\beta h(\sigma)}.
\label{Z}
\end{equation*}
Furthermore, we let 
\begin{itemize}
\item $Z_{zero}(G)$ denote the Potts model partition function with the Hamiltonian of Equation~\eqref{z.ham};

\item  $Z_{ext}(G)$ denote the Potts model partition function with the Hamiltonian of Equation~\eqref{e.ham2}.
\end{itemize}
\end{definition}

\subsection{Vertex-weighted graphs}

We use standard notation for graphs. $E(G)$ and $V(G)$ denote the edge set and vertex set, respectively, of a graph $G$.  If $A \subseteq E(G)$, then $n(A)$, $r(A)$, and $k(A)$ are, respectively, the nullity, rank, and number of components of the {\em spanning subgraph}, $(V(G),A)$, of $G$ with edge set $A$.   

Here, a {\em vertex-weighted graph} consists of a graph $G$, with vertex set $V(G)=\{v_1, v_2, \ldots , v_n\}$ equipped with a  weight function $\omega$ mapping $V(G)$ into a  commutative semigroup. The {\em weight} of the vertex $v_i$ is the value $\omega (v_i)$. We will often write $\omega_i$ rather than $\omega (v_i)$, and $\boldsymbol{x}$ for an indexed set of variables. 
An {\em edge-weighted} graph $G$ is a graph equipped with a mapping $\gamma$ from its edge set  $E(G)$ to a set $\boldsymbol{\gamma}:=\{\gamma_e\}_{e\in E(G)}$.  We use the standard convention  $\gamma:e\mapsto \gamma_e$, for $e\in E(G)$. In addition, at times we  will  write ``$\gamma_i$'' instead of ``$\gamma_{e_i}$''.  In this paper, we assume that all graphs are edge-weighted and vertex-weighted.

If $G$ is a vertex-weighted graph with weight function $\omega$, and $e$ is an edge of $G$, then $G-e$ is the graph obtained from $G$ by deleting the edge $e$ and leaving the weight function unchanged. If $e$ is any non-loop edge of $G$, then $G/e$ is the graph obtained from $G$ by contracting the edge $e$ and changing  the vertex weight function as follows: if $v_i$ and $v_j$ are the vertices incident to $e$, and $v$ is the vertex of $G/e$ created by the contraction, then $ \omega(v) =  \omega(v_i) +\omega(v_j) $. See Figure~\ref{f.dc}.  Loops are not contracted. 

\begin{figure}
\[\begin{array}{ccccc}
\includegraphics[width=2.5cm]{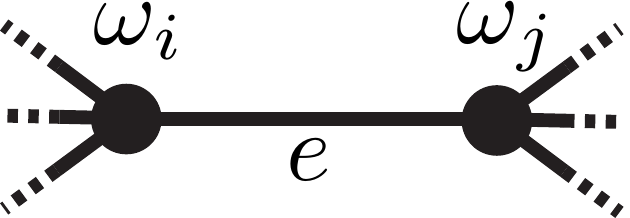} & \hspace{1cm}&\includegraphics[width=2.5cm]{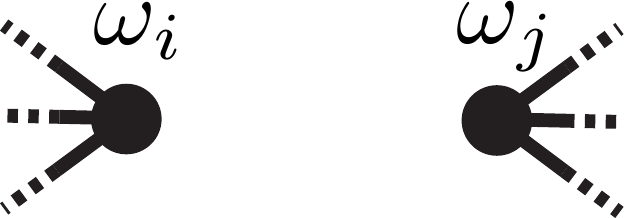} &\hspace{1cm}&\includegraphics[width=2.5cm]{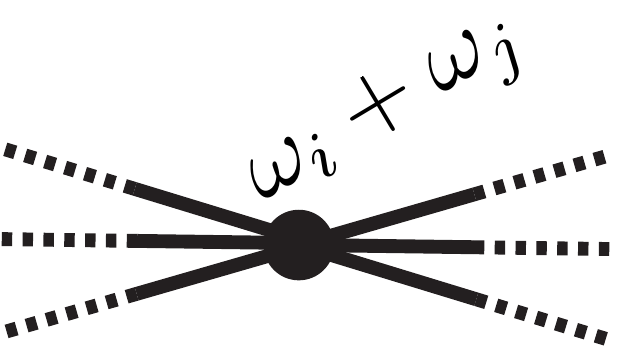}  \\
G && G-e && G/e
\end{array}
\]
\caption{Deletion and contraction of an edge  in a vertex-weighted graph.}
\label{f.dc}
\end{figure}

\section{Expansions for the Potts partition function}\label{s.exp}

\subsection{ Relating $Z_{ext}$ and $Z_{zero}$ via connected partitions}
Let $G$ be a vertex-weighted graph with vertex set $V(G)$, and  let $\pi=(V_1, \ldots V_k)$ be a partition of $V(G)$.  Each vertex set $V_i$ induces a subgraph $G_i$ of $G$. (The subgraph $G_i$ consists of the vertex set $V_i$ and all of the edges of $G$ for which both ends lie in $V_i$.) 
We say that a partition $\pi$ of $G$ is a {\em connected partition} if each induced subgraph $G_i$, for $i=1,\ldots ,k$, is connected. We  let $\p(G)$   denote the set of connected partitions of $V(G)$. We let $\C(\pi)$ denote the set of subgraphs $\{G_1, \ldots , G_k\}$ induced by a connected partition $\pi$.

Finally, we let $[x^i]$ denote the coefficient operator, so that if $P=\sum_{j=1}^k a_j x^j$, then $[x^i]P=a_i$.

We can now state the first of our main results. This result expresses the  Potts model partition function with site dependent external fields, in terms of the zero field Potts model partition function.
\begin{theorem}\label{t.exp2} 
Let $G$ be a graph equipped with a  field vector  $\boldsymbol{M}_i=(M_{i,1}, \ldots , M_{i,q})\in \mathbb{C}^q$ at each vertex $v_i$. Then the Potts model partition function admits  the following expansion:
\[Z_{ext}(G)= \sum_{\pi\in \p(G)} 
X(\pi)
\prod_{H\in\C(\pi)} [q^1]\left(Z_{zero}(H;q, \{\e{\beta J_e} -1  \}_{e\in E(H)} )\right),
\]
where, for each connected partition $\pi$, 
$\C(\pi)$ is the set of subgraphs of $G$ induced by $\pi$;
  $X(\pi) = X_{\boldsymbol{V}_1} \cdots X_{\boldsymbol{V}_k}$ for   $X_{\boldsymbol{M}}=      \sum_{\alpha=1}^q \e{\beta M_{\alpha}}$, and 
$\boldsymbol{V}_i$ is the sum of the weights, $\boldsymbol{M}_i$, of all of the vertices  in $V_i$. 
\end{theorem}
We defer the proof of this theorem until Section~\ref{s.for}.

\subsection{Spanning tree and spanning forest expansions for $Z_{ext}$}\label{ss.stsf}
Our other two expansions for $Z_{ext}$ require some additional notation on trees, forests and edge activity.

A {\em spanning forest} of $G$  is an acyclic spanning subgraph. We let  $\mathcal{F}(G)$ denote the set of spanning forests of $G$.
If $G$ is a connected graph then a {\em spanning tree} of $G$ is   a connected spanning forest. Following a standard abuse of notation from the theory of the Tutte polynomial, if $G$ is a graph with components $G_1, \ldots ,G_{k(G)}$, then we say that $T$ is a {\em spanning tree} of $G$ if  the restriction of $T$ to each component of $G$ is a spanning tree of that component, {\em i.e.}  $T\cap G_i$ is a spanning tree of $G_i$ for each $i=1, \ldots , k(G)$.  We let  $\mathcal{T}(G)$ denote the set of spanning trees of $G$.

Suppose  $T\in\mathcal{T}(G)$ is a spanning tree of $G$. 
If $e\in T$, then the {\em cut} defined by $e$ is the set $\{ f\in E(G) | (T-e)\cup f \in \mathcal{T}(G)  \}$; and if $e\notin T$ the {\em cycle defined by $e$} is the unique cycle in $T\cup e$.   
If, in addition, the edges of $G$ have been given an ordering $\{e_1, \ldots , e_m\}$, then  edge $e\in T$ is {\em internally active} if it is the smallest edge in the cut defined by $e$, and is {\em internally inactive} otherwise; and an edge  $e\notin T$ is {\em externally active} if it is the smallest edge in the cycle defined by $e$, and is {\em externally inactive} otherwise.  Activities are defined analogously when  $F\in \mathcal{F}(G)$ is a spanning forrest of $G$ (simply replace ``$\mathcal{T}(G)$'' by ``$\mathcal{F}(G)$'', and ``tree'' by ``forest'' in the above).
We will let $\IA(G,T)$, $\II(G,T)$, $\EA(G,T)$ and $\EI(G,T)$ denote the set  of  internally active, internally inactive, externally active, and externally inactive edges of $G$ with respect to the spanning tree $T$  or spanning forest $T$. When the graph $G$ is clear from context,  we will sometimes write $\II(T)$ instead of $\II(G,T)$, and do similarly for the other sets above.
Also note that since $\II(G,T)$ and $\IA(G,T)$ are both  sets of edges of the spanning subgraph $T$, we can regard $\II(G,T)$ and $\IA(G,T)$ as  sets of edges of $T$ itself. This allows us to write expressions such as $T/\II(G,T)$.

Our second main result is a spanning forest expansion for the Potts model  partition function with site dependent external fields.
\begin{theorem}\label{t.exp3} 
Let $G$ be a graph equipped with a  field vector  $\boldsymbol{M}_i=(M_{i,1}, \ldots , M_{i,q})\in \mathbb{C}^q$ at each vertex $v_i$. Then the Potts model partition function with Hamiltonian~\eqref{e.ham2} admits the spanning forest expansion 
\[
Z_{ext}(G)=  \sum_{F \in  \F(G) } X(F)   \left(\prod_{e\in F}(\e{\beta J_e} -1)\right)\left(\prod_{e\in \EA(G,F)}\e{\beta J_e}\right),
\]
where 
$X(F) = X_{\boldsymbol{F}_1} \cdots X_{\boldsymbol{F}_k}$, for   $X_{\boldsymbol{M}}=      \sum_{\alpha=1}^q \e{\beta M_{\alpha}}$, and $\boldsymbol{F}_i$ is the sum of the weights, $\boldsymbol{M}_i$, of all of the vertices  in the $i$-th component, $F_i$, of the spanning forest.

\end{theorem}
The proof of this theorem appears in Section~\ref{s.for}.

Our final main result is a spanning tree expansion for the Potts model  partition function with site dependent external fields.  
It is well-known that the Tutte polynomial can be defined as a sum over spanning trees (see, for example, the exposition in \cite{Bol98}). Using this spanning tree expansion for the Tutte polynomial together with Equation~\eqref{e.potts}, one can obtain a spanning tree expansion for zero-field Potts model partition function with constant edge interactions:
\begin{multline}\label{e.zsp}  Z_{zero}(G;q,e^{\beta J}-1)   =  \\
q^{k(G)} (e^{\beta J}-1)^{|V(G)|-k(G)} 
\sum_{T \in \T(G) }   \left(    \frac{q + e^{\beta J}-1}{e^{\beta J}-1}  \right)^{| \IA(T) |}  \left(e^{\beta J}\right)^{| \EA(T) |} .  \end{multline}
(Note that this spanning tree expansion can be extended to a spanning tree expansion for the zero-field Potts model partition function with variable edge interactions using Equation~\eqref{e.mvtp} and Corollary~\ref{c.treeMV} below.) Here we extend the expansion shown in Equation~\eqref{e.zsp} to $Z_{ext}(G)$:
\begin{theorem}\label{t.exp1}
Let $G$ be a graph equipped with a  field vector  \mbox{$\boldsymbol{M}_i=(M_{i,1}, \ldots , M_{i,q})\in \mathbb{C}^q$} at each vertex $v_i$. Then the Potts model partition function admits the spanning tree expansion
\begin{equation}\label{e.exp1}  Z_{ext}(G)=  \sum_{T \in \T(G) }   \left( \prod_{e\in \II(G,T)}  ( \e{\beta J_{e}}-1) \right) 
 \left(  \prod_{e\in \EA(G,T)}  \e{\beta J_{e}}  \right) 
  Z_{ext}(   T/ \II(G,T)  ).    \end{equation}
\end{theorem}
The proof of this theorem appears in Section~\ref{s.trees}.

\section{The $\V$-polynomial and the Potts model}\label{s.v}

\subsection{The Tutte polynomial and the Potts model}
 There are  seminal relations between the Tutte polynomial, its extensions, and the zero-field Potts model partition function $Z_{zero}(G)$.    
 The {\em multivariate Tutte polynomial}, $Z_T(G)$, (see \cite{Sok05} and \cite{Tra89}) of a graph $G$ with edge-weights $\boldsymbol{\gamma}:=\{\gamma_e\}_{e\in E(G)}$ is  defined by
  \[ Z_T(G; \theta, \boldsymbol{\gamma} ) := \sum_{A\subseteq E(G)}  \theta^{k(A)} \prod_{e\in A} \gamma_e.    \]
  It is well known that the Potts model partition function with the zero-field Hamiltonian and variable edge-interaction energies of Equation~\eqref{z.ham} can be recovered from the multivariate Tutte polynomial: 
\begin{equation}\label{e.mvtp}
Z_{zero}(G)=Z_T( G;  q, \{ \e{\beta {J_{e}}} -1 \}_{e\in E(G)}). 
\end{equation}
(See \cite{FK72} for the early stages of this theory and later exposition in, for example, \cite{Bax82,BE-MPS10,Bol98,Wel93,WM00}.) 
Furthermore, if the edge-interaction energies are constant with $J_e=J$, for each edge $e$, then the Potts model partition function can be recovered from the Tutte polynomial as in Equation~\eqref{e.potts}. 
As mentioned in the introduction, this connection has resulted in a remarkable interactions between the areas of combinatorics and statistical mechanics, particularly for computational complexity and the study of the zeros of these polynomials. (See, for example, \cite{BE-MPS10,GJ07,Roy09,Sok05,WM00}  for surveys of these results.) 
The equivalence of the Tutte polynomial and the Potts model partition function assumes the absence of an external field, and $Z_{ext}$ can not be expressed in terms of the (multivariate) Tutte polynomial.
(Some weighting strategies have been developed to study non-zero  fields in a graph colouring framework. See \cite{CS09,CS10,SX10}.)

\subsection{The $\V$-polynomial and the Potts model}
In \cite{EMM11}, Ellis-Monaghan and Moffatt used a definition of contraction that incorporates vertex weights to assimilate a Hamiltonian of the generic form of Equation~\eqref{e.ham2} into the theory of the Tutte-Potts connection.
To do this, a new graph polynomial, called the  $\V$-polynomial, was introduced. The  $\V$-polynomial generalizes Noble and Welsh's $W$-polynomial of \cite{NW99}, and extends the Tutte polynomial by incorporating vertex weights and adapting contraction to accommodate them. It was shown that the variable field Potts model partition function (with its many specializations) is an evaluation of the $\V$-polynomial, and hence is a polynomial with a deletion-contraction reduction and a Fortuin-Kasteleyn type representation.  
\begin{definition}[\cite{EMM11}]\label{mvw}
Let $S$ be a  commutative semigroup, let $G$ be a graph equipped with vertex weights $\boldsymbol{\omega}:=\{ \omega_i \} \subseteq S$  and edge weights $\boldsymbol{\gamma}:= \{\gamma_e\}$, and let $\boldsymbol{x}=\{x_k\}_{k\in S}$ be a set of commuting variables. Then the {\em $\V$-polynomial}, $\V(G) = \V(G, \omega; \boldsymbol{x}, \boldsymbol{\gamma}) \in \mathbb{Z} [ \{\gamma_e\}_{e\in E(G)}, \{x_k\}_{k\in S}  ] $, of the vertex- and edge-weighted graph $G$, 
 is defined recursively by: 
\begin{enumerate}
\item $ \V(G) = \V(G-e)+\gamma_e  \V(G/e)$, if $e$ is a non-loop edge of $G$;
\item $\V(G)=(\gamma_e+1)\V(G-e)$, if $e$ is a loop;
\item $ \V(E_m)= \prod_{i=1}^{m} x_{\omega_i}$, if $E_m$ consists of $m$ isolated vertices of weights $\omega_1, \ldots , \omega_m$.
\end{enumerate}
\end{definition}

We will use the following  state sum expansion for the  $\V$-polynomial later.
\begin{theorem}[\cite{EMM11}]\label{t.vsum}
$\V(G)$ can be represented as a sum over spanning subgraphs:
\[ \V(G)=\sum_{A\subseteq E(G)}   x_{c_1}   x_{c_2}\cdots   x_{c_{k(A)}}   \prod_{e\in A} \gamma_e ,   \]
where $c_i$ is the sum of the weights of all of the vertices in the $i$-th connected component of the spanning subgraph $(V(G), A)$.
\end{theorem}

The main result of \cite{EMM11} was the recovery of $Z_{ext}(G)$ as an evaluation of the  Tutte-type graph polynomial $\V(G)$, extending the classical connection between $T$, $Z_T$ and $Z_{zero}$, and providing  a deletion-contraction reduction for  the partition function $Z_{ext}(G)$.
\begin{theorem}[\cite{EMM11}]\label{ZV2} Let $G$ be a graph equipped with a  field vector  $\boldsymbol{M}_i=(M_{i,1}, \ldots , M_{i,q})\in \mathbb{C}^q$ at each vertex $v_i$, and
let $h(\sigma)$ be the Hamiltonian given in Equation~\eqref{e.ham2}.
Then 
\[  Z_{ext}(G)=\V\left( G,\omega ; \;\{X_{\boldsymbol{M}}\}_{\boldsymbol{M}\in \mathbb{C}^q}    ,\;   \{  \e{\beta J_{i,j}}-1   \}_{\{i,j\}\in E(G)}      \right),\]
where the vertex weights  are given by $\omega (v_i) =\boldsymbol{M}_i$ 
 and, for any $\boldsymbol{M} = (M_1, \ldots , M_q) \in \mathbb{C}^q$, 
$ X_{\boldsymbol{M}}=  \sum_{\alpha=1}^q \e{\beta M_{\alpha}}$.
\end{theorem}

We note that Theorem~\ref{ZV2} unifies an important segment of Potts model theory and brings previously successful combinatorial machinery, including complexity results, to bear on a wider range of statistical mechanics models. Furthermore, 
Theorem~\ref{ZV2} gives a Fortuin-Kasteleyn-type representation for $Z_{ext}(G)$ and shows  that the partition function $Z_{ext}(G)$ is in fact a polynomial. See \cite{EMM11} for details.

\section{Spanning trees and the Potts model}\label{s.trees}

In this  section we prove Theorem~\ref{t.exp1}, which gives a spanning tree expansion for the Potts model partition function $Z_{ext}(G)$. This result generalizes the spanning tree expansions for $Z_{zero}$ that come from the Tutte polynomial $T(G)$ and the multivariate Tutte polynomial $Z_T(G)$. The expansion for the partition function $Z_{ext}$ will follow from the spanning tree expansion for the $\V$-polynomial given in the following theorem.

\begin{theorem}\label{vexp}Let $G$ be a graph equipped with vertex weights $\boldsymbol{\omega}:=\{ \omega_i \}$  and edge weights $\boldsymbol{\gamma}:= \{\gamma_e\}$. Then
\[\V(G, \omega;\textit{\textbf{x}}, \boldsymbol{\gamma})=
  \sum_{T \in \T(G) }      \left(  \prod_{e\in \II(G,T)}  \gamma_e  \right)  
   \left(  \prod_{e\in \EA(G,T)}  (\gamma_e+1)  \right)  
  \V  \left(     T/ \II(G,T) \right). 
 \]
 \end{theorem}

\begin{proof}
For ease of reading, we  set \[\alpha(G,T):=  (\prod_{e\in \II(G, T)}  \gamma_e)(  \prod_{e\in \EA(G,T)}  (\gamma_e+1) ) .\]

We will prove the theorem by induction on the number of edges of $G$. 

The assertion is easily verified when $G$ has no edges.

Now suppose $G$ has ordered edges $\{e_1, \ldots , e_m\}$, $m\geq 1$, and assume the assertion holds for all graphs with fewer than $m$ edges. Note that the order of the edges of  $G$ induces an order of the edges of $G-e_m$ and of $G/e_m$. We will use these induced orders throughout the proof. We will consider three cases: when $e_m$ is ordinary, a bridge, or a loop.

\noindent\underline{Case 1:} If $e_m$ is ordinary, then
\begin{multline}\label{p1}
 \sum_{T \in \T(G) }    \alpha(G,T)\,  \V  \left(   T/ \II(G,T)   \right) \\=  
  \mathop{\sum_{T \in  \T(G) }}_{e_m\notin T}   \alpha(G,T)\,  \V  \left(   T/ \II(G,T)   \right)
+\mathop{\sum_{T \in  \T(G) }}_{e_m\in T}        \alpha(G,T)\,  \V  \left(   T/ \II(G,T)   \right) .
\end{multline}
Focussing on the first term on the right-hand side of \eqref{p1}, we have that $T\in \T(G)$ with $e_m\notin T$ if and only if $T\in \T(G-e_m)$.  Since $e_m$ is not a loop and is the largest edge, $e_m\in \EI(G,T)$. In addition,   
for $1\leq i<m$, we have that  $e_i \in \EA(G,T)$   if and only if $e_i \in \EA(G-e_m,T)$, (this is since  the cycles defined by $e_i$ in $G$ and $G-e_m$, with respect to $T$, are identical). We also have that $e_i \in \II(G,T)$   if and only if $e_i \in \II(G-e_m,T)$, (this is since the ordinary edge $e_m$ is the largest edge and so is never the smallest edge in the cut defined by $e_i$ in $G$ and in $G-e_m$).
When $e_m\notin T$, we then have $\EA(G,T) =   \EA(G-e_m, T)$, and $\II(G,T) =   \II(G-e_m, T)$. Thus $\alpha(G,T)= \alpha(G-e_m, T)$ and $  T/ \II(G,T) =  T/ \II(G-e_m, T)$.

Focussing on the second term on the right-hand side of \eqref{p1}, we have that $T\in \T(G)$ with $e_m\in T$ if and only if $T/e_m\in \T(G/e_m)$.  
Also, since $e_m$ is not a bridge and is the largest edge, $e_m\in \II(G,T)$. In addition,   
for $1\leq i<m$, we have that  $e_i \in \EA(G,T)$   if and only if $e_i \in \EA(G/e_m,T/e_m)$, (this is since  the cycles defined by $e_i$ in $G$ and $G-e_m$, both must contain an edge smaller than $e_m$). Furthermore, we have that $e_i \in \II(G,T)$   if and only if $e_i \in \II(G/e_m,T/e_m)$, (this is since  the cuts defined by $e_i$ in $G$ and $G/e_m$, with respect to the two trees, are identical).
When $e_m\in T$, we then have $\EI(G,T) =   \EI(G/e_m, T/e_m)  $, and $\II(G,T) =   \II(G/e_m, T/e_m)\cup \{e_m\}$. 
Thus $\alpha(G,T)=  \gamma_m \alpha(G/e_m, T/e_m)$ and $  T/ \II(G,T) =  T/( \II(G/e_m, T/e_m)\cup \{e_m\} )  =(T/e_m)/ \II(G/e_m, T/e_m)$.

Using the above observations, we can then write \eqref{p1} as
\begin{multline*}  \sum_{T' \in  \T(G-e_m) }    \alpha(G-e_m,T')\,   \V  \left(    T' /\II(G-e_m, T')  \right) \\
+  \gamma_m \sum_{T'' \in  \T(G/e_m) }    \alpha(G/e_m, T'')\,   \V  \left(     T''/ \II(G/e_m,T'') \right) ,\end{multline*}
which, by the inductive hypothesis, is equal to 
$\V(G-e_m)+ \gamma_m\V(G/e_m) =  \V(G)$, as required.


\noindent\underline{Case 2:} 
If $e_m$ is bridge, then $\V(G)=\V(G-e_m)+\gamma_m \V(G/e_m$), which, by the inductive hypothesis, is equal to 
\begin{multline*}   \sum_{T' \in  \T(G-e_m) }    \alpha(G-e_m, T')\,   \V  \left(     T'/\II(G-e_m,T')   \right) \\
+  \gamma_m \sum_{T'' \in  \T(G/e_m) }    \alpha(G/e_m, T'')\,   \V  \left(     T''/\II(G-e_m,T'')\right)  .\end{multline*}
Since $e_m$ is a bridge, we have that $T$ is a spanning tree of $G$ if and only if $T-e_m$ is a spanning tree of $G-e_m$. Also, $T$ is a spanning tree of $G$ if and only if $T/e_m$ is a spanning tree of $G/e_m$.

Also since $e_m$ is a bridge, it is the unique edge in its cut, and so $e_m$ is not in $\EA(G,T)$ nor in $\II(G,T)$. Furthermore, since $e_m$ is  the largest edge, checking the activity of an edge $e_i$, $1\leq i<m$, in $G$ with respect to $G$ is the same as checking the activity of $e_i$  in $G-e_m$ with respect to $G-e_m$, which is also the same as checking the activity of $e_i$  in $G/e_m$ with respect to $G/e_m$. Thus we have $\EA(G,T)=\EA(G-e_m, T-e_m )=\EA(G/e_m, T/e_m)$ and $\II(G,T)=\II(G-e_m, T-e_m )=\II(G/e_m, T/e_m)$.
 Thus $\alpha(G,T)=  \alpha(G-e_m, T-e_m) = \alpha(G/e_m, T/e_m)$ and $  T/ \II(G,T) =  (T-e_m)/ \II(G-e_m, T-e_m)=  (T/e_m)/ \II(G/e_m, T/e_m)$.

Using this, we may write the above sum as
\begin{multline*}    \sum_{T \in  \T(G) }   [ \alpha(G-e_m, T-e_m )\,    \V  \left(    (T-e_m)/ \II(G-e_m, T-e_m)\right)  \\
+  \gamma_m  \,    \alpha(G/e_m, T/e_m  )\,  \V  \left(      (T/e_m)/ \II(G/e_m, T/e_m)
\right) ]\\
= \sum_{T \in  \T(G) }     \alpha(G, T )  \, \left(  \V  \left(     [T/ \II(G, T)]-e_m  \right)  \right.
 \left.+  \gamma_m   \,  \V  \left(     [T/ \II(G, T)]/e_m  \right)\right)
,
\end{multline*}
which, via deletion-contraction, is just 
$   \sum_{T \in  \T(G) }       \alpha( G, T )  \,   \V  \left(    T/ \II(G, T) \right)$,   
as required.


\noindent\underline{Case 3:} Finally, if $e_m$ is a loop, then 
$\V(G)=(\gamma_m+1)\V(G-e_m)$, which, by the inductive hypothesis, equals
\begin{equation}\label{p4} (\gamma_m+1)  \sum_{T' \in  \T(G-e_m) }     \alpha( G-e_m, T' )  \,    \V  \left(    (G-e_m)/\II(G-e_m,T')  \right) .  \end{equation}
Clearly, $T\in\T(G)$  if and only if $T\in \T(G-e_m)$. 
Since $e_m$ is a loop, it is externally active with respect to every spanning tree $T$ of $G$.  Also, for $1\leq i<m$, the activity of $e_i$ in $G$ with respect to $T$ is identical to the activity of $e_i$ in $G-e_m$ with respect to $T$. 
Thus $\alpha(G,T)=  (\gamma_m+1)\alpha(G-e_m, T-e_m)$ and $  T/ \II(G,T) =  T/ \II(G-e_m, T-e_m)$.

Using this, the expression \eqref{p4}  can be written as   
\[ \sum_{T \in  \T(G) }  \alpha(G,T)  \,   \V  \left(   T/ \II(G,T)   \right) ,\]  
and the result follows.

\end{proof}

It was shown in \cite{EMM11} that the multivariate Tutte polynomial $Z(G)$ can be recovered from the  $\V$-polynomial by setting $x_i=\theta$ for each $i$: 
\begin{equation}\label{vz} 
\V(G, \omega; x_i=\theta , \boldsymbol{\gamma})  =Z_T(G; \theta , \boldsymbol{\gamma}).
\end{equation}
We can use this fact together with Theorem~\ref{vexp} to recover Traldi's spanning tree expansion for the multivariate Tutte polynomial from \cite{Tra89}  (see also \cite{BR99}). We will use this expansion in the proof of Theorem~\ref{forestV}.
\begin{corollary}[Traldi \cite{Tra89}]  \label{c.treeMV} 
Let $G$ be a  graph equipped with edge weights $\boldsymbol{\gamma}:= \{\gamma_e\}$. Then 
\[
Z_T(G;\theta,\boldsymbol{\gamma})=  \theta^{k(G)} \sum_{T\in \T(G)}  \;\;\prod_{e\in \II(G,T)}\gamma_{e} \; \prod_{e\in \IA(G,T)}(\gamma_{e}+\theta)\;\prod_{e\in \EA(G,T)}(\gamma_{e}+1),
  \]
  where $k(G)$ denotes the number of  components of $G$.
\end{corollary}
\begin{proof}
By Equation~\eqref{vz} and Theorem~\ref{vexp},
\[ Z_T(G;\theta,\boldsymbol{\gamma})= \sum_{T \in \T(G) }       \prod_{e\in \II(G,T)}  \gamma_e  \;  
     \prod_{e\in \EA(G,T)}  (\gamma_e+1)  \; 
  Z_T  \left(     T/ \II(G,T) ;\theta,\boldsymbol{\gamma}\right).  \]
However,  $T/ \II(G,T) $ is just a forest with $k(G)$ components whose edge set is in bijection with $\IA(G,T)$, so $Z_T  \left(     T/ \II(G,T) \right)= \theta^{k(G)}\prod_{e\in \IA(G,T)}(\theta+\gamma_{e})$, and the result follows.
\end{proof}
We note that the well-known spanning tree expansion for the Tutte polynomial \[T(G;x,y)=\sum_{T \in \T(G) } x^{| \IA(T) |}y^{|\EA(T)|}\] can easily be recovered from Corollary~\ref{c.treeMV}. In addition,  Theorem~\ref{vexp} gives spanning tree expansions for Noble and Welsh's polynomials $U(G)$ and $W(G)$ from \cite{NW99}. 

\begin{remark}
Observe that the proof of Corollary~\ref{c.treeMV}  indicates why the expression $  \V  \left(     T/ \II(G,T) \right)$ appears in the spanning tree expansion for the $\V$-polynomial, but $  Z_T$   does not appear  in the expansion for the multivariate Tutte polynomial (and similarly for the Tutte polynomial): while the values of the Tutte polynomial and the multivariate Tutte polynomial on trees are easily computable, this is not the case for the $\V$-polynomial.
\end{remark}

\begin{proof}[Proof of Theorem \ref{t.exp1}]
By Theorem \ref{ZV2}, we have  \[Z(G)=\V\left( G,\omega ; \;\{X_{\boldsymbol{M}}\}_{\boldsymbol{M}\in \mathbb{C}^q}    ,\;   \{  \e{\beta J_{e}}-1   \}_{e\in E(G)}      \right).\] The result then follows by applying Theorem~\ref{s.trees}.  

\end{proof}

\section{Spanning forests and the Potts model}\label{s.for}
In this  section we prove the expansions stated in Theorems~\ref{t.exp2} and~\ref{t.exp3}. 
Our approach to both of these theorems is to find a corresponding expansion for the $\V$-polynomial and then to apply Theorem~\ref{ZV2}. Both of the expansions for the $\V$-polynomial and the proof of these expansions are extensions of results due to Noble and Welsh   on the $W$-polynomial in \cite{NW99}. Our first result shows that the polynomial  $\V(G)$ admits an expansion in terms of the multivariate Tutte polynomial $Z_T$.

\begin{theorem}\label{statesumV} 
Let $G$ be a graph equipped with vertex weights $\boldsymbol{\omega}:=\{ \omega_i \}$  and edge weights $\boldsymbol{\gamma}:= \{\gamma_e\}$. Then
\[
\V(G, \omega;\textit{\textbf{x}}, \boldsymbol{\gamma})=\sum_{\pi\in \p(G)}x(\pi)\prod_{H\in \mathcal{C}(\pi)} [q^1] \left( Z_T(H;q,\boldsymbol{\gamma}) \right),
  \]
where $x(\pi):= x_{c_1}\cdots x_{c_k}$, and $c_i$ is the sum of the vertex weights in the $i$-th block of the partition $\pi$.
\end{theorem}

\begin{proof}%
We begin with the state sum expansion of $\V(G)$ from Theorem~\ref{t.vsum}: 
\[
\boldsymbol{V}(G;\boldsymbol{x},\boldsymbol{\gamma})=\sum_{A\subseteq E(G)}x_{c_1}x_{c_2}\cdots x_{c_{k(A)}}\prod_{e\in A}{\gamma_{e}}.\]
Letting $\Pi(A)$ denote the partition of $V(G)$ whose blocks $V_i$ consist of the vertices in  the connected components of the spanning subgraph with edge set  $A$, we have 
\begin{equation}\label{e.pexp1}
\boldsymbol{V}(G;\boldsymbol{x},\boldsymbol{\gamma})=
\sum_{A\subseteq E(G)}x_{c_1}x_{c_2}\cdots x_{c_{k(A)}}\prod_{e\in A}{\gamma_{e}}=\sum_{\pi\in \p(G)}x(\pi)\sum_{A:\Pi(A)=\pi}\;\prod_{e\in A}{\gamma_{e}},
  \end{equation}
  where the sum is over all connected partitions $\pi$ of $V(G)$.

Suppose that the partition $\pi \in \p(G)$ has blocks $V_1, \ldots , V_k$ and the connected components of the subgraph induced by $\pi$ are $G_1, \ldots , G_k$ where $V_i=V(G_i)$.
Then
\begin{multline*}
\sum_{A:\Pi(A)=\pi} \;\prod_{e\in A}{\gamma_{e}}   = 
\mathop{\mathop{\sum_{A_1 \sqcup \cdots \sqcup A_k \subseteq E(G_1) \sqcup \cdots \sqcup E(G_k)}}_{ (V_i,A_i) \text{ connected} } }_{A_i\subseteq E(G_i)} 
\;\;
\prod_{e\in A_1 \sqcup \cdots \sqcup A_k} \gamma_e 
\\  = \prod_{i=1}^k \mathop{\sum_{A_i  \subseteq E(G_i)}}_{ (V_i,A_i) \text{ connected} } \prod_{e\in A_i}\gamma_e \;\; = \prod_{i=1}^k [q^1]Z(G_i;q,\boldsymbol{\gamma})
= \prod_{H\in \C(\pi)} [q^1]Z(H;q,\boldsymbol{\gamma}).
\end{multline*}
Substituting this into Equation~\eqref{e.pexp1}  gives the result.
%
\end{proof}

\begin{proof}[Proof of Theorem~\ref{t.exp2}]
The result follows immediately from Theorems~\ref{ZV2} and~\ref{statesumV}, and Equation~\eqref{e.mvtp}.
\end{proof}


Using the sum over connected partitions from Theorem~\ref{statesumV}, we will obtain an expression for $\V(G)$ as a sum over spanning forests of $G$. We will use this spanning forest expansion to prove Theorem~\ref{t.exp3}.
\begin{theorem}\label{forestV} 
Let $G$ be a graph equipped with vertex weights $\boldsymbol{\omega}:=\{ \omega_i \}$  and edge weights $\boldsymbol{\gamma}:= \{\gamma_e\}$. Then
\[\V(G, \omega;\textit{\textbf{x}}, \boldsymbol{\gamma})=
  \sum_{F \in  \F(G) } x(F)   \; \prod_{e\in E(F)}\gamma_{e}\;\prod_{e\in \EA(G,F)}(1+\gamma_{e})
\]
where $x(F):= x_{c_1}\cdots x_{c_k}$, where $c_i$is the sum of the vertex weights in the $i$-th component of $F$.
\end{theorem}

\begin{proof}
By Theorem~\ref{statesumV}, we have
\begin{equation}\label{e.exp31}
\V(G;\textit{\textbf{x}}, \boldsymbol{\gamma})=\sum_{\pi\in \p(G)}x(\pi)
\prod_{H\in \mathcal{C}(\pi)} [q^1] \left( Z_T(H;q,\boldsymbol{\gamma}) \right).
  \end{equation}
 Also, by Corollary~\ref{c.treeMV}, when $H$ is connected, 
\begin{multline*}
[q^1]Z_{T}(H;q,\boldsymbol{\gamma})=[q^1]\left( q \sum_{T\in \T(H)}\;\prod_{e\in \II(H,T)}\gamma_{e}\;\prod_{e\in \IA(H,T)}(q+\gamma_{e}) \;\prod_{e\in \EA(H,T)}(1+\gamma_{e}) \right) \\
\\
= \sum_{T\in \T(H)}\; \prod_{e\in \II(H,T)}\gamma_{e}\;  \prod_{e\in \IA(H,T)}\gamma_{e}\; \prod_{e\in \EA(H,T)}(1+\gamma_{e})
= \sum_{T\in \T(H)}\;\prod_{e\in E(T)}\gamma_{e} \;\prod_{e\in \EA(H,T)}(1+\gamma_{e}).
  \end{multline*}
  Substituting this into Equation~\eqref{e.exp31} gives
\begin{multline*}
  \V(G;\textit{\textbf{x}}, \boldsymbol{\gamma})=
  \sum_{\pi\in \p(G)}x(\pi) 
  \prod_{H\in \mathcal{C}(\pi)} \sum_{T\in \T(H)}\;\prod_{e\in E(T)}\gamma_{e}\; \prod_{e\in \EA(T)}(1+\gamma_{e}) \\=
  \sum_{\pi\in \p(G)}x(\pi)  \mathop{\sum_{F \in  \F(G) }}_{ \Pi(F)=\pi}   
\;\prod_{e\in E(F)}\gamma_{e}\;\prod_{e\in \EA(F)}(1+\gamma_{e})\;
\\=  \sum_{F \in  \F(G) } x(F)   \;\prod_{e\in E(F)}\gamma_{e}\;\prod_{e\in \EA(F)}(1+\gamma_{e}),
\end{multline*}
as required.
\end{proof}

We can now prove the second of our main results.
\begin{proof}[Proof of Theorem~\ref{t.exp3}] 
The theorem follows immediately from Theorems~\ref{ZV2} and~\ref{forestV}.
\end{proof}

\section{Applications to  common models}\label{s.extra}
In this final section we provide some specializations to a few other  common models. 

The first example  is used in models in which  a system with an external field in which one particular spin (the first, without loss of generality) is preferred (see for example the surveys \cite{Ber05,Wu82}).  In the standard applications to magnetism in  models of magnetism in an external field
 $\beta = 1/(\kappa T)$, where $T$ is the temperature of the system, and where $\kappa
= 1.38 \times 10^{-23} $ Joules/Kelvin is the Boltzmann constant. 
Here again, variable interaction energies and  field values are allowed. 
\begin{theorem}\label{t.extra1}
Let $G$ be a graph in which a  complex value $z_i$ is associated to each vertex $v_i$,  let $h$ be the  Hamiltonian
 \begin{equation}\label{e.h71} h(\sigma) = - \sum_{ \{ i,j \} \in E }  J_{i,j}  \delta(\sigma_i, \sigma_j)  - \sum_{v_i\in V(G)} z_i  \delta(1,  \sigma_i),
 \end{equation}
 and let $Z$ denote the Potts model partition function arising from this Hamiltonian.
\begin{enumerate}
\item If, for $z\in \mathbb{C}$, $X_z=\e{\beta z} +q-1$, then  
\[Z(G)= \sum_{\pi\in \p(G)} 
X(\pi)
\prod_{H\in\C(\pi)} [q^1]\left(Z_{zero}(H;q, \{\e{\beta J_e} -1  \}_{e\in E(H)} )\right),
\]
where, for each connected partition $\pi$, 
 $X(\pi)= X_{c_1} \cdots X_{c_k}$, where $c_j$ is the sum of the weights, $z_i$, of the vertices in the $j$-th block of the partition $\pi$, and $\C(\pi)$ is the set of subgraphs of $G$ induced by $\pi$.

\item If, for $z\in \mathbb{C}$, $X_z=\e{\beta z} +q-1$, then 
\[
Z(G)=  \sum_{F \in  \F(G) } X(F)   \left(\prod_{e\in F}(\e{\beta J_e} -1)\right)\left(\prod_{e\in \EA(G,F)}\e{\beta J_e}\right),
\]
where 
$X(F)= X_{c_1} \cdots X_{c_k}$, where $c_j$ is the sum of the weights, $z_i$, of the vertices in the $j$-th component of the spanning forrest F.

\item  $Z$ admits the spanning tree expansion
\[Z(G)=  \sum_{T \in \T(G) }   \left( \prod_{e\in \II(G,T)}  ( \e{\beta J_{e}}-1) \right) 
 \left(  \prod_{e\in \EA(G,T)}  \e{\beta J_{e}}  \right) 
  Z(   T/ \II(G,T)  ). \]

\end{enumerate}

\end{theorem}
\begin{proof}
The theorem follows easily from Theorems \ref{t.exp2}, \ref{t.exp3}, and \ref{t.exp1} upon noting that the Hamiltonian~\eqref{e.h71} can be recovered from the Hamiltonian~\eqref{e.ham2} by setting $\boldsymbol{M}_i=(z_i, 0, \ldots, 0)$ for each $i$, and that then each vector $\boldsymbol{M}_i$ is described by $z_i$ and so we may reindex the expressions using complex numbers.
\end{proof}

The following corollary contains the restriction of Theorem~\ref{t.extra1} to models with constant interaction energies and a constant, non-zero  field.
\begin{corollary}
Let $G$ be a graph,  $h$ be the  Hamiltonian
 \begin{equation}\label{e.h72} h(\sigma) = - J\sum_{ \{ i,j \} \in E }    \delta(\sigma_i, \sigma_j)  -H \sum_{v_i\in V(G)}   \delta(1,  \sigma_i),
 \end{equation}
 and let $Z$ denote the Potts model partition function arising from this Hamiltonian.
\begin{enumerate}
\item We have
\[Z(G)=   \sum_{\pi\in \p(G)} 
X(\pi)
\prod_{K\in\C(\pi)} [q^1]\left(Z_{zero}(K;q, \e{\beta J} -1   )\right),
\]
where, for each connected partition $\pi$, 
$X(\pi)=\prod_{i=1}^{k(\pi)} ( \e{H|\pi_i|} +q-1  )$, where $k(\pi)$ is the number of blocks in $\pi$, and $|\pi_i|$ is the size of its $i$-th block.

\item We have
\[
Z(G)=  \sum_{F \in  \F(G) } X(F)   \left(\prod_{e\in F}(\e{\beta J_e} -1)\right)\left(\prod_{e\in \EA(G,F)}\e{\beta J_e}\right),
\]
where, 
$X(F)=\prod_{i=1}^{k(F)} ( \e{H|V(F_i)|} +q-1  )$, where  $F_i$ is the $i$-th component of the spanning forrest $F$.

\item  $Z$ admits the spanning tree expansion
\[Z(G)=  \sum_{T \in \T(G) }   \left( \prod_{e\in \II(G,T)}  ( \e{\beta J_{e}}-1) \right) 
 \left(  \prod_{e\in \EA(G,T)}  \e{\beta J_{e}}  \right) 
  Z(   T/ \II(G,T)  ). \]
\end{enumerate}
\end{corollary}

Our final specialization is to  random field Ising model (RFIM) which is used to study disordered states.  It has a random  field in that the $z_i$'s are randomly chosen local magnetic fields that each affect only a single site.  In \cite{EMM11}, it was shown that the RFIM (which is essentially the $q=2$ Potts model) can be expressed as an evaluation of the $\V$-polynomial. The RFIM is also particularly interesting from the point of view of graph polynomials as  there have been two very recent studies of the Ising model from a graph theoretical perspective.  \cite{AM09} treats the Ising model with constant interaction energies and constant (but non-zero) magnetic field as a graph invariant and explores graph theoretical properties encoded by it.  \cite{WF} on the other hand, creates a new graph polynomial, which has a deletion-contraction reduction for non-loop edges, and which is, up to a pre-factor and change of variables, equivalent to the RFIM. 

The  RFIM takes spin values in $\{-1,1\}$. We will let  $\tau$  denote a state for the RFIM, which is a map $\tau:V(G)\rightarrow \{-1,+1\}$. As usual we set $\tau_i:=\tau(v_i)$. Also we will let $\mathcal{T}(G)$ be the set of states for the Ising model.
\begin{theorem}
Let $G$ be a graph with a vertex weight $z_i\in \mathbb{C}$ associated to each vertex $v_i$.
The RFIM  given by
\[ 
h(\tau) = - J\sum_{ \{ i,j \} \in E }  {\tau_i \tau_j} - \sum_{i \in V(G)}{z_i \tau_i},\quad \text{and} \quad Z(G)=\sum_{\tau\in \mathcal{T}(G)} \e{-\beta h(\tau)}
\]
admits the following expansions.
\begin{enumerate}
\item We have
\[  Z(G)= \e{-\beta \eta(G)}
 \sum_{F \in  \F(G) } x(F)   \;  (\e{2\beta J}-1)^{|E(F)|}\;\prod_{e\in \EA(G,F)}\e{2\beta J},
\]
where for any $z\in \mathbb{C}$, 
$  x_{z}=    \e{2z} + \e{4z} $;  $\eta(G)= J|E(G)| +3 \sum_{i \in V(G)} z_i $,and
$x(F):= x_{c_1}\cdots x_{c_k}$, for $c_j$  the sum of the vertex weights in the $j$-th component of $F$.

\item We have, for $\eta(G)= J|E(G)| +3 \sum_{i \in V(G)} z_i $,
\[Z(G)=   \e{-\beta \eta(G)}   \sum_{T\in\T(G)}   (\e{2\beta J}-1)^{|\II(T)|}\,  \e{2\beta J |\EA(T)| +\beta \eta(T/\II(T)) }  Z(T/\II(T)).  \]
\end{enumerate}

\end{theorem}
\begin{proof}
By Theorem~6.3 of \cite{EMM11}, 
$ Z(G)= \e{-\beta \eta(G)}  \V\left( G,\omega  ; \;\{x_{z}\}_{z\in \mathbb{C}}    ,\;    \e{2\beta J}-1       \right)
$,
where, for any $z\in \mathbb{C}$, 
$  x_{z}=    \e{2z} + \e{4z} $, and $\omega(v_i)= z_i$. 
The first item follows by applying Theorem~\ref{statesumV}, and the second item follows
by applying Theorem~\ref{vexp}.


\end{proof}




\end{document}